\documentclass[12pt]{amsart}%
\usepackage{amsfonts}
\usepackage{mathtools}
\usepackage{amsmath}
\usepackage{amssymb}
\usepackage{amsthm}
\usepackage{newlfont}
\usepackage{graphicx}
\usepackage{amssymb,amsmath,xcolor,graphicx,xspace,colortbl, rotating}
\usepackage{enumitem}
\usepackage{textcomp}%
\providecommand{\U}[1]{\protect\rule{.1in}{.1in}}
\graphicspath{{proof_latest teX_graphics/}{proof_latest teX_tcache/}{proof_latest teX_gcache/}}
\DeclareGraphicsExtensions{.pdf,.eps,.ps,.png,.jpg,.jpeg}
\newtheorem {theorem}{Theorem}[section]

\newtheorem {claim}[theorem]{Claim}

\newtheorem {lemma}[theorem]{Lemma}

\numberwithin{equation}{section}
\begin{document}
\title[Hamiltonian Stationary]{Interior Schauder Estimates for the Fourth Order Hamiltonian Stationary Equation in two dimensions}
\author{Arunima Bhattacharya AND Micah Warren }

\begin{abstract}
We consider the Hamiltonian stationary equation for all phases in dimension
two. We show that solutions that are $C^{1,1}$ will be smooth and we also
derive a $C^{2,\alpha}$ estimate for it.

\end{abstract}
\maketitle

\section{ Introduction}

In this paper, we study the regularity of the Lagrangian Hamiltonian
stationary equation, which is a fourth order nonlinear PDE. Consider the
function $u:B_{1}\rightarrow\mathbb{R}$ where $B_{1}$ is the unit ball in
$\mathbb{R}^{2}$. The gradient graph of $u$, given by $\{(x,Du(x))|x\in
B_{1}\}$ is a Lagrangian submanifold of the complex Euclidean space. The
function $\theta$ is called the Lagrangian phase for the gradient graph~and is
defined by
\[
\theta=F(D^{2}u)=Im\log\det(I+iD^{2}u)
\]
or equivalently,
\begin{equation}
\theta=\sum_{i}\arctan(\lambda_{i}) \label{P}%
\end{equation}
where $\lambda_{i}$ represents the eigenvalues of the Hessian.

The nonhomogenous special Lagrangian equation is given by the following second
order nonlinear equation
\begin{equation}
F(D^{2}u)=f(x). \label{SL}%
\end{equation}

The Hamiltonian stationary equation is given by the following fourth order
nonlinear PDE%

\begin{equation}
\Delta_{g}\theta=0\label{HS}%
\end{equation}
where $\Delta_{g}$ is the Laplace-Beltrami operator, given by:
\[
\Delta_{g}=\sum_{i,j=1}^{2}\frac{\partial_{i}(\sqrt{detg}g^{ij}\partial_{j}%
)}{\sqrt{detg}}%
\]
and $g$ is the induced Riemannian metric from the Euclidean metric on
$\mathbb{R}^{4}$, which can be written as
\[
g=I+(D^{2}u)^{2}.
\]

Recently, Chen and Warren \cite{CW} proved that in any dimension, a
$C^{1,1\text{ }}$solution of the Hamiltonian stationary equation will be
smooth with uniform estimates of all orders if the phase $\theta\geq
\delta+(n-2)\pi/2,$ or, if the bound on the Hessian is small. In the two
dimensional case, using \cite{CW}'s result, we get uniform estimates for $u$
when $\left\vert \theta\right\vert \geq\delta>0$ (by symmetry). In this paper,
we consider the Hamiltonian stationary equation for all phases in dimension
two without imposing a smallness condition on the Hessian or on the range of
$\theta$, and we derive uniform estimates for $u$, in terms of the $C^{1,1}$
bound which we denote by $\Lambda$. We write $||u||_{C^{1,1}(B_{1})}= ||Du||_{C^{0,1}(B_{1})}=\Lambda$. Our main results are the following:

\begin{theorem}
Suppose that $u\in C^{1,1}(B_{1})\cap W^{2,2}(B_{1})$ and satisfies (\ref{HS})
on $B_{1}\subset\mathbb{R}^{2}$. Then u is a smooth function with interior
H\"{o}lder estimates of all orders, based on the $C^{1,1}$ bound of $u$.
\end{theorem}

\begin{theorem}
Suppose that $u\in C^{1,1}(B_{1})\cap W^{2,2}(B_{1})$ and satisfies (\ref{SL})
on $B_{1}\subset\mathbb{R}^{2}$. If $\theta\in C^{\alpha}(B_{1})$, then there
exists $R=R(2,\Lambda,\alpha)<1$ such that $u\in C^{2,\alpha}(B_{R})$ and
satisfies the following estimate
\begin{equation}
|D^{2}u|_{C^{\alpha}(B_{R})}\leq C_{1}(||u||_{L^{\infty}(B_{1})}, \Lambda, |\theta|_{C^{\alpha}(B_{1})}).\label{E}%
\end{equation}

\end{theorem}

Our proof goes as follows: we start by applying the De Giorgi-Nash theorem to
the uniformly elliptic Hamiltonian stationary equation (\ref{HS}) on $B_{1}$
to prove that $\theta\in C^{\alpha}(B_{1/2})$. Next we consider the
non-homogenous special Lagrangian equation (\ref{SL}) where~$\theta\in
C^{\alpha}(B_{1/2}).$ Using a rotation of Yuan \cite{Yuan2002} we rotate the
gradient graph so that the new phase $\bar{\theta}$ of the rotated gradient
graph satisfies $\left\vert \bar{\theta}\right\vert \geq\delta>0$. Now we apply \cite{CC} to the new
potential $\bar{u}$ of the rotated graph to obtain a $C^{2,\alpha}$ interior
estimate for it. On rotating back the rotated gradient graph to our original
gradient graph, we see that our potential $u$ turns out to be $C^{2,\alpha}$
as well. A computation involving change of co-ordinates gives us the
corresponding $C^{2,\alpha}$ estimate, shown in (\ref{E}). Once we have a
$C^{2,\alpha}$ solution of (\ref{HS}), smoothness follows by \cite[Corollary
5.1]{CW}.

In two dimensions, solutions to the second order special Lagrangian equation
\[
F(D^{2}u)=C
\]
enjoy full regularity estimates in terms of the potential $u$
\cite{WarrenYuan2}. \ For higher dimensions, such estimates fail
\cite{WangYuan} for $\theta=C$ with $\left\vert C\right\vert <(n-2)\pi/2$.

\section{ Proof of theorems:}

We first prove Theorem 1.2, followed by the proof of Theorem 1.1. We prove
Theorem 1.2 using the following lemma.

\begin{lemma}
Suppose that $u\in C^{1,1}(B_{1})\cap W^{2,2}(B_{1})$ satisfies (\ref{SL}) on
$B_{1}\subset\mathbb{R}^{2}$. Suppose
\begin{equation}
0\leq\theta(0)<(\pi/2-\arctan\Lambda)/4\label{C}.%
\end{equation} 
 If $\theta\in C^{\bar{\alpha}}(B_{1}%
)$, then there exists $0<\alpha<\bar{\alpha}$ and $C_{0}$ such that
\[
|D^{2}u(x)-D^{2}u(0)|\leq C_{0}(||u||_{L^{\infty}(B_{1})}, \Lambda,|\theta|_{C^{\alpha}(B_{1})})\ast|x|^{\alpha}.
\]

\end{lemma}

\begin{proof}
Consider the gradient graph~$\{(x,Du(x))|x\in B_{1}\}$ where $u$ has the
following Hessian bound
\[
-\Lambda I_{n}\leq D^{2}u\leq\Lambda I_{n}%
\]
~a.e. where it exists.

Define $\delta$ as
\[
\delta=(\pi/2-\arctan\Lambda)/2>0.
\]
Since by (\ref{C}) we have $0\leq\theta(0)<\delta/2$,
there exists $R^{\prime}(\delta,|\theta|_{C^{\bar{\alpha}}})>0$ such that
\[
|\theta(x)-\theta(0)|<\delta/2
\]
for all $x\in B_{R^{\prime}}\subseteq B_{1}$. This implies for every~$x$
in~$B_{R^{\prime}}$ for which $D^{2}u$ exists, we have
\[
\delta>\theta>\theta(0)-\delta/2.
\]
So now we rotate the gradient graph $\{(x ,D u (x))\vert x \in B_{R^{ \prime}%
}\}$ downward by an angle of $\delta$.

Let the new rotated co-ordinate system be denoted by $(\bar{x},\bar{y})$
where~
\begin{align}
\bar{x}  &  =\cos(\delta)x+\sin(\delta)Du(x)\label{R_0}\\
\bar{y}  &  =-\sin(\delta)x+\cos(\delta)Du(x). \label{R_1}%
\end{align}

On differentiating $\bar{x}$ (\ref{R_0}) with respect to $x$ we see that
\[
\frac{d\bar{x}}{dx}=\cos(\delta)I_{n}+\sin(\delta)D^{2}u(x)\leq\cos
(\delta)I_{n}+\Lambda\sin(\delta)I_{n}%
\]
Thus
\[
\cos(\delta)I_{n}-\Lambda\sin(\delta)I_{n}\leq\frac{d\bar{x}}{dx}\leq
\cos(\delta)I_{n}+\Lambda\sin(\delta)I_{n}.
\]
To obtain Lipschitz constants so that
\begin{equation}
\frac{1}{L_{2}}I_{n}\leq\frac{d\bar{x}}{dx}\leq L_{1}I_{n} \label{L}%
\end{equation}
let%

\begin{align*}
L_{1}  &  =\cos(\delta)+\Lambda\sin(\delta)\\
L_{2}  &  =\max\{%
%TCIMACRO{\QOVERD{\vert}{\vert}{1}{\cos(\sigma)I_{n}+D^{2}u(x)\sin(\sigma)}}%
%BeginExpansion
\genfrac{\vert}{\vert}{}{}{1}{\cos(\delta)I_{n}+D^{2}u(x)\sin(\delta)}%
%EndExpansion
|x\in B_{R^{\prime}}\}.
\end{align*}
To find the value of $L_{2}$, we see that in $B_{R^{\prime}}$ we have the
following:\\
let  $\min\{\theta_{1},\theta_{2}\}\geq-A$  where $A=\arctan{\Lambda}$.
\begin{align*}
\cos(\delta)I_{n}+\sin(\delta)D^{2}u(x)  &  \geq\cos(\delta)-\sin(\delta)\tan(A)\\
&  =\cos(\delta)(1-\tan(\delta)\tan(A))\\
& =\cos(\delta)\frac{\tan(\delta)+\tan(A)}{\tan(\delta+A)}\\
& =\cos(\delta)\frac{\tan(\delta)+\tan(A)}{\tan(\frac{\pi/2-A}{2}+A)}\\
& =\cos(\delta)\frac{\tan(\delta)+\tan(A)}{\tan(\pi/2-\delta)}.
\end{align*}
This shows that 
\[
\frac{1}{L_{2}}=\cos(\delta)\frac{\tan(\delta)+\tan(A)}{\tan(\pi/2-\delta)}.
\]
Clearly $1/L_{2}$ is positive. 

Now, by \cite[Prop 4.1]{CW} we see that there exists a function~$\bar{u}$ such
that~
\[
\bar{y}=D_{\bar{x}}\bar{u}(\bar{x})
\]
where
\begin{equation}
\bar{u}(x)=u(x)+\sin\delta\cos\delta\frac{|Du(x)|^{2}-|x|^{2}}{2}-\sin
^{2}(\delta)Du(x)\cdot x \label{DU}%
\end{equation}
\newline defines $\bar{u}$ implicity in terms of $\bar{x}$ (since $\bar{x}$ is
invertible). Here $\bar{x}$ refers to the rotation map (\ref{R_0}).

Note that
\[
\bar{\theta}(\bar{x})-\bar{\theta}(\bar{y})=\theta(x)-\theta(y)
\]
which implies that $\bar{\theta}$ is also a $C^{\bar{\alpha}}$ function
\[
\frac{|\bar{\theta}(\bar{x}_{1})-\bar{\theta}(\bar{x}_{2})|}{|\bar{x}_{1}%
-\bar{x}_{2}|^{\alpha}}=\frac{|\theta(x_{1})-\theta(x_{2})|}{|x_{1}%
-x_{2}|^{\bar{\alpha}}}\ast\frac{|x_{1}-x_{2}|^{\bar{\alpha}}}{|\bar{x}%
_{1}-\bar{x}_{2}|^{\bar{\alpha}}}%
\]
\newline thus,%
\[
|\bar{\theta}|_{C^{\bar{\alpha}}(B_{r_{0}})}\leq L_{2}^{\bar{\alpha}}%
|\theta|_{C^{\bar{\alpha}}(B_{R^{\prime}})}.
\]

Let $\Omega=\bar{x}(B_{R^{ \prime}})$. Note that $B_{r_{0}}\subset\Omega$
where $r_{0} =R^{ \prime}/2L_{2}.$ So our new gradient graph is $\{(\bar{x}
,D_{\bar{x}} \bar{u} (\bar{x}))\vert\bar{x} \in\Omega\}$. The function
$\bar{u}$ satisfies the equation~
\[
F(D_{\bar{x}}^{2} \bar{u}) =\bar{\theta} (\bar{x})
\]
~in~$B_{r_{0}}$ where $\bar{\theta} \in
C^{\alpha} (B_{r_{0}})$. Observe that on $B_{r_{0}}$ we have
\begin{equation*}
\bar{\theta}=\theta-2\delta<\delta-2\delta=-\delta<0
\end{equation*} as $\theta<\delta$ on $B_{R^{\prime}}$.

\begin{claim}:
If $\bar{\vert \theta }\vert  >\delta $, then $F(D^{2}\bar{u}) =\bar{\theta }$ is a solution to a uniformly elliptic concave equation.  \label{CL}%
\end{claim}
\begin{proof}
The proof follows from \cite[lemma 2.2]{CPW16} and also from \cite[pg 24]{CW}.
 
\end{proof}

 Now using \cite[Corollary 1.3]{CC} we get interior Schauder estimates for $\bar{u}$:

\begin{equation}
|D^{2}\bar{u}(\bar{x})-D^{2}\bar{u}(0)|\leq C(||\bar{u}||_{L^{\infty}(B_{r_{0}/2})}+|\bar{\theta}|_{C^{\alpha}%
(B_{r_{0}/2})})\label{SC1}%
\end{equation}
for all $\bar{x}$ in $B_{r_{0}/2}$ where $C=C(\Lambda,\alpha)$.
This is our $C^{2,\alpha}$ estimate for $\bar{u}$.

Next, in order to show the same Schauder type inequality as (\ref{SC1}) for
$u$ in place of $\bar{u}$, we establish relations between the following pairs:

\begin{enumerate}
[label=(\roman*)]

\item oscillations of the Hessian of $D^{2}u$ and $D^{2}\bar{u}$

\item oscillations of $\theta$ and $\bar{\theta}$

\item the supremum norms of $u$ and $\bar{u}$ .
\end{enumerate}

We rotate back to our original gradient graph by rotating up by an angle
of $\delta$ and consider again the domain $B_{R^{\prime}}(0).$
This gives us the following
relations:
\begin{align}
x =\cos(\delta) \bar{x} -\sin(\delta) D_{\bar{x}} \bar{u} (\bar{x})\nonumber\\
y = \sin(\delta) \bar{x} +\cos(\delta) D_{\bar{x}} \bar{u} (\bar{x}).
\label{R}%
\end{align}
This gives us:
\begin{align*}
\frac{dx}{d \bar{x}} =\cos(\delta) I_{n} -\sin(\delta) D_{\bar{x}}^{2} \bar{u}
(\bar{x})\\
D_{\bar{x}} y =\sin(\delta) I_{n} +\cos(\delta) D_{\bar{x}}^{2} \bar{u}
(\bar{x}).
\end{align*}
\newline So we have
\[
D_{x}^{2}u (x) =D_{\bar{x}} y \frac{d\bar{x}}{dx}\newline=[\sin(\delta) I_{n}
+\cos(\delta) D_{\bar{x}}^{2} \bar{u}(\bar{x})][\cos(\delta) I_{n}
-\sin(\delta) D_{\bar{x}}^{2} \bar{u} (\bar{x})]^{-1}.
\]
The above expression is well defined everywhere because $D_{\bar{x}}^{2}%
\bar{u}(\bar{x}) <\cot(\delta) I_{n}$ for all $\bar{x} \in B_{r_{0}}$.
\newline

Note that we have $\cos(\delta)I_{n}-D_{\bar{x}}^{2}\bar{u}(\bar{x}%
)\sin(\delta)\geq\frac{1}{L_{1}}$, since
\[
\frac{dx}{d\bar{x}}=\cos(\delta)I_{n}-\sin(\delta)D_{\bar{x}}^{2}\bar{u}%
(\bar{x})=\left(  \frac{d\bar{x}}{dx}\right)  ^{-1}\geq\frac{1}{L_{1}}I_{n}%
\]
by (\ref{L}).

Next,
\begin{align}
D_{x}^{2}u(x)-D_{x}^{2}u(0)  &  =[\sin(\delta)I_{n}+\cos(\delta)D_{\bar{x}%
}^{2}\bar{u}(\bar{x})][\cos(\delta)I_{n}-\sin(\delta)D_{\bar{x}}^{2}\bar
{u}(\bar{x})]^{-1}\nonumber\\
&  -[\sin(\delta)I_{n}+\cos(\delta)D_{\bar{x}}^{2}\bar{u}(0)][\cos
(\delta)I_{n}-\sin(\delta)D_{\bar{x}}^{2}\bar{u}(0)]^{-1}. \label{Hessian}%
\end{align}

For simplification of notation we write
\begin{align*}
D_{\bar{x}}^{2}\bar{u}(\bar{x})  &  =A\\
D_{\bar{x}}^{2}\bar{u}(0)  &  =B\\
\cos(\delta)  &  =c,\sin(\delta)=s.
\end{align*}
~Noting that $[sI_{n}+cA]$ and $[cI_{n}-sA]^{-1}$ commute with each other we
can write (\ref{Hessian}) as the following equation
\begin{align*}
D_{x}^{2}u(x)-D_{x}^{2}u(0)  &  =\\
&  \lbrack cI_{n}-sB]^{-1}[cI_{n}-sB][sI_{n}+cA][cI_{n}-sA]^{-1}-\\
&  \lbrack cI_{n}-sB]^{-1}[sI_{n}+cB][cI_{n}-sA][cI_{n}-sA]^{-1}.
\end{align*}
~\ Again we see that
\[
\lbrack cI_{n}-sB][sI_{n}+cA]-[sI_{n}+cB][cI_{n}-sA]=A-B.
\]
This means
\[
D_{x}^{2}u(x)-D_{x}^{2}u(0)=[cI_{n}-sB]^{-1}[A-B][cI_{n}-sA]^{-1}.
\]
We have already shown that
\[
|cI_{n}-sA|\geq\frac{1}{L_{1}}%
\]
\newline which implies
\[
|cI_{n}-sA|^{-1}\leq L_{1}.
\]
~\ Thus we get~%
\begin{align}
|D_{x}^{2}u(x)-D_{x}^{2}u(0)|  &  \leq L_{1}^{2}|D_{\bar{x}}^{2}\bar{u}%
(\bar{x})-D_{\bar{x}}^{2}\bar{u}(0)|.\nonumber\\
&  \leq CL_{1}^{2}(||\bar{u}||_{L^{\infty
}(B_{r_{0}/2})}+|\bar{\theta}|_{C^{\alpha}(B_{r_{0}/2})})|\bar{x}|^{\alpha}\nonumber\\
&  \leq CL_{1}^{2+\alpha}(||\bar{u}||_{L^{\infty}(B_{r_{0}/2%
})}+|\bar{\theta}|_{C^{\alpha}(B_{r_{0}/2})}|x|^{\alpha}\label{bigalpha}%
\end{align}
where $L_{1}$ is the Lipschitz constant of the co-ordinate change map. This
implies
\begin{equation}
\frac{1}{L_{1}^{\alpha+2}}|D_{x}^{2}u(x)|_{C^{\alpha}(B_{R})}\leq|D_{\bar{x}%
}^{2}u(\bar{x})|_{C^{\alpha}(B_{r_{0}/2})}. \label{LH}%
\end{equation}

Recall from (\ref{DU}) that
\[
\bar{u}(x)=u(x)+g(x).
\]
This shows
\begin{align}
||\bar{u}(\bar{x})||_{L^{\infty}(B_{r_{0}/2})}=||\bar{u}(x)||_{L^{\infty}%
(\bar{x}^{-1}(B_{r_{0}/2}))}\leq||\bar{u}(x)||_{L^{\infty}(B_{R^{\prime}})} \nonumber \\
\leq||u(x)||_{L^{\infty}(B_{R^{\prime}})}+||g||_{L^{\infty}(B_{R^{\prime}})}.%
\label{U}%
\end{align}
Note that~
\begin{equation}
||g||_{L^{\infty}(B_{R})}\leq R||Du||_{L^{\infty}(B_{R})}+\frac{1}{2}%
[R^{2}+||Du||_{L^{\infty}(B_{R})}^{2}] \label{ggg}%
\end{equation}

and combining (\ref{LH}), (\ref{U}), (\ref{ggg}) with (\ref{bigalpha}) we get
\begin{align*}
&  |D_{x}^{2}u(x)-D_{x}^{2}u(0)|\\
&  \leq CL_{1}^{\alpha+2}\left\{
\begin{array}
[c]{c}%
||u||_{L^{\infty}(B_{R^{\prime}})}\\
R||Du||_{L^{\infty}(B_{R})}+\frac{1}{2}[R^{2}+||Du||_{L^{\infty}(B_{R})}%
^{2}]\\
+L_{2}^{\alpha}r_{0}|\theta|_{C^{\alpha}(B_{R^{\prime}})}%
\end{array}
\right\}  \left\vert x\right\vert ^{\alpha}.
\end{align*}
This proves the Lemma. \ 
\end{proof}

\begin{proof}
[Proof of Theorem 1.2]First note that the lemma gives a H\"{o}lder norm on any
interior ball, by a rescaling of the form
\[
u_{\rho}(x)=\frac{u(\rho x)}{\rho^{2}}%
\]
for values of $\rho>0$ and translation of any point to the origin. \ Consider
the gradient graph~$\{(x,Du(x))|x\in B_{1}\}$ where $u$ satisfies
\[
F(D^{2}u)=\theta
\]
on $B_{1}$ and $\theta\in C^{\bar{\alpha}}(B_{1})\ \text{.}\ $ Then there
exists a ball of radius $r$ inside $B_{1}$ on which $osc\theta<\delta/4$ where
$\delta$ is as defined in Lemma 2.1. \newline Now this means that either we
have $\theta(x)<\delta/2$ in which case, by the above lemma we see that $u\in
C^{2,\alpha}(B_{r})$ satisfying the given estimates; or we have $\theta
(x)>\delta/4$ in which case $u\in C^{2,\alpha}(B_{r})$ with uniform estimates,
by claim (\ref{CL}) and \cite[Corollary 1.3]{CC}\newline.
\end{proof}

\begin{proof}
[Proof of Theorem 1.1]Since $u\in C^{1,1}(B_{1})\cap W^{2,2}(B_{1})$ satisfies
the uniformly elliptic equation
\[
\Delta_{g}\theta=0,
\]
by the De Giorgi-Nash Theorem we have that $\theta\in C^{\alpha}%
(B_{1/2}).$ This means that u satisfies
\[
F(D^{2}u)=\theta.
\]
By Theorem 1.2 we see that $u\in C^{2,\alpha}(B_{r})$ where $r<1/2$.
\ Smoothness follows by \cite[Corollary 5.1]{CW}. \ 
\end{proof}

\bibliographystyle{amsalpha}
\bibliography{Hamiltonian_Stationary}

\end{document}